\documentclass[12pt]{article}
\usepackage{mathrsfs}
\usepackage{amsmath}
\usepackage{amsthm}
\usepackage{bbm}
\usepackage{subfloat}
\usepackage{subfigure}
\usepackage{fancyhdr,graphicx}
\usepackage{amsfonts}
\usepackage{amssymb}
\usepackage{latexsym,bm}
\usepackage{newlfont}
\usepackage{float}
\usepackage{pifont}
\usepackage{adjustbox}
\usepackage{mathtools}
\newtheorem{thm}{Theorem}[section]
\newtheorem{Lemma}[thm]{Lemma}
\newtheorem{cor}[thm]{Corollary}
\newtheorem{pro}[thm]{Proposition}
\newtheorem{Remark}[thm]{Remark}

\newtheorem{ex}[thm]{Example}

\makeatletter
\long\def\@makecaption#1#2{%
 \vskip\abovecaptionskip
  \sbox\@tempboxa{{#1.}\quad #2}%
   \ifdim \wd\@tempboxa >\hsize
    { #1.}\quad #2\par
     \else
  \global \@minipagefalse
   \hb@xt@\hsize{\hfil\box\@tempboxa\hfil}%
   \fi
   \vskip\belowcaptionskip}
\makeatother

\usepackage{latexsym, bm}
\title{{Minimum forcing numbers of perfect matchings of circular  and prismatic  graphs}\footnote{This
 work is supported by NSFC (grant no. 12271229).}}

\author{ Qiaoyun Shi, Heping Zhang\footnote{The
corresponding author.}\\
\small{School of Mathematics and Statistics, Lanzhou University,
Lanzhou, Gansu 730000, P. R. China}\\
\small{E-mail addresses: shiqy21@lzu.edu.cn, zhanghp@lzu.edu.cn}}

\date{}

\begin{document}
\maketitle
\begin{abstract}
Let $G$ be a graph  with a perfect matching. Denote by $f(G)$ the minimum size of a matching  in $G$ that is uniquely extendable to a perfect matching in $G$. Diwan (2019) used linear algebra to prove that for the $d$-hypercube $Q_d$ ($d\geq 2)$, $f(Q_d)=2^{d-2}$, thus settling a conjecture  of Pachter and Kim (1998).  Recently, Mohammadian generalized this method to prove a general result: for a bipartite graph $G$ on $n$ vertices, if  $G$ admits an involutory weighted adjacency matrix $A$ over a field $F$,
 then $f(G\Box K_2)=\frac{n}{2}$, where $\square$ denotes the Cartesian product of two graphs. In this paper we obtain  $f(G\Box C_{2k})=n$ when a bipartite graph $G$ on $n$ vertices admits an involutory weighted adjacency matrix $A$ over a field $F$ of  characteristic not 2, for all integers $k\ge2$. Moreover, we demonstrate that this method can also be applied to some nonbalanced bipartite graphs $G$  when graphs $G$  admit a weighted bi-adjacency matrix with orthogonal rows.
\vskip 0.1cm


\end{abstract}

 \section{Introduction}
All graphs in this paper are finite and simple. For a graph $G$, let $V(G)$ denote its vertex set and $E(G)$ its edge set. The order of  $G$ is the  number of its vertices. An edge subset $S\subseteq E(G)$ is called a matching if no two edges in $S$ share a common end-vertex.  A matching of $G$  is  perfect  if it covers every  vertex of $G$. For a perfect matching $M$ of  $G$, a subset $S$ of $M$ is called a {\it forcing set} if $S$ is contained in no other perfect matchings of $G$. The \textit{forcing number} of $M$ in $G$, denoted by $f(G,M)$, is the smallest size of a forcing set  for $M$. The \textit{minimum forcing number} of $G$, denoted by $f(G)$, is defined as the minimum value of \(f(G, M)\) over all perfect matchings $M$ of $G$. In general, computing the minimum forcing number of a bipartite graph is an NP-complete problem \cite{Adams,Afshani}. For two surveys and some recent work on matching forcing, see  \cite{che,Liu,zhang2025}.

The concept of the forcing number of a perfect matching of a graph originated  in theoretical chemistry, where it corresponds to the ``innate degree of freedom" of a  Kekul\'e structure of  polycyclic aromatic compounds,  introduced by  Klein and Randi\'c \cite{Klein1}.  The term ``forcing" was later adopted in graph theory by Harary et al. \cite{Harary}. 
 The forcing idea has since appeared in various combinatorial structures \cite{Donovan2003}, such as graph colorings,  Latin squares, dominating sets, and block designs.


Let $Q_d$ denote the $d$-dimensional hypercube graph, whose vertex set is $\{0,1\}^d$, and where  two vertices are adjacent if and only if their corresponding vectors differ in exactly one
coordinate. Diwan \cite{Diwan} determined  the minimum forcing number of $Q_n$ by a linear algebra technique, settling a conjecture originally posed by Pachter and Kim \cite{Pather}
and later revised  by Riddle \cite{Riddle}. Riddle had previously verified the statement for even $d$ by a combinatorial approach.

\begin{thm}[Diwan \cite{Diwan}]\label{diwan}  $f(Q_d)=2^{d-2}$ for any integer $d\geq 2$.
\end{thm}

Diwan's method is based on the following two simple lemmas. The latter relates to the corank of a matrix.

\begin{Lemma}[Diwan \cite{Diwan}]\label{diwan-lem}
Let $G$ be a graph of order $n$ with a perfect matching. If the largest order of an induced subgraph of $G$ having a unique perfect matching is $k$, then the forcing number of every perfect matching in $G$ is at least $(n-k)/2$.
\end{Lemma}

 Let $F$ be a field. In particular, $\mathbb R$ and $\mathbb C$ denote the real and complex fields, respectively. A {\it weighted adjacency matrix} $A$ of a graph $G$ is a  matrix over $F$, with   rows and columns  indexed by $V(G)$, such that  the $(u, v)$-entry  is non-zero if and only if $u$ and $v$ are adjacent in $G$.   If $G$ is a bipartite graph with bipartition $(X,Y)$, then a {\it weighted bi-adjacency matrix} $B$ of $G$ is a submatrix of  $A$ whose rows and columns are indexed by partite sets $X$ and $Y$, respectively. The following result was implied and used in \cite{Diwan, Mohammadian}.


\begin{Lemma}[Diwan \cite{Diwan}, Mohammadian \cite{Mohammadian}] \label{lem}
Let $G$ be a bipartite graph with a perfect matching. If $A$ (resp. $B$) denotes a weighted (resp. bi-)adjacency matrix of $G$ over a field $F$, then
\begin{equation}f(G)\ge \frac{1}{2}(|V(G)| -\operatorname{rank}(A))=\frac{1}{2}\operatorname{corank}(A), \mbox{ and } f(G)\ge \frac{1}{2}(|V(G)| -2\,\operatorname{rank}(B))=\operatorname{corank}(B).
\end{equation}
\end{Lemma}

%
%
%

Diwan \cite{Diwan} constructed a weighted  adjacency matrix $A$ of $Q_d$ such that its inverse  $A^{-1}$ is also a weighted  adjacency matrix of $Q_d$, and both matrices have integer  entries, which allowed him  to prove Theorem \ref{diwan}. Huang \cite{Huang} constructed another weighted adjacency matrix of $Q_d$ with entries in $\{0,1,-1\}$ whose rows are orthogonal to confirm the famous Sensitivity Conjecture in information theory.


The Cartesian product of two graphs $G$ and $H$, denoted by $G\Box H$, is the graph with vertex set $V(G)\times V(H)$, and two vertices $(g_1,h_1)$ and $(g_2,h_2)$ are adjacent if and only if either $g_1=g_2$ and $h_1h_2 \in E(H)$ or $h_1=h_2$ and $g_1g_2 \in E(G)$. Thus, the $d$-hypercube $Q_d$ can be expressed as the $d$-fold Cartesian product of $K_2$ (a complete graph with two vertices), and $Q_d=Q_{d-1}\Box K_2$ for $d\geq 2$.

Diwan \cite{Diwan} asked whether this algebraic technique could be applied to compute the
minimum forcing number of perfect matchings in other bipartite graphs.  Mohammadian \cite{Mohammadian} later generalized  Diwan's method  to the {\it prism} of a bipartite graph $G$, i.e., $G\Box K_2$.  A matrix $A$ is called {\it involutory} if $A^2 = I$ (equivalently, $A^{-1} = A$). Let  $\mathcal{I}_F (G)$ be the set of weighted adjacency matrices of $G$ over $F$ that are involutory.

\begin{thm}[Mohammadian \cite{Mohammadian}]\label{Mohammadian} Let $G$ be a bipartite graph with $n$ vertices. If $\mathcal{I}_F (G)\ne \emptyset$, then
\begin{equation}f(G\Box K_2)=\frac{n}{2}.\end{equation}
\end{thm}


Note that for a bipartite graph $G$, $\mathcal{I}_F (G)\ne \emptyset$ holds if and only if $G$ admits a weighted bi-adjacency matrix $B$ such that  the pattern of non-zero entries of $B^{-1}$ matches that of  its transpose  $B^{\top}$ and the bipartition of $G$ is balanced.


 In this paper, we mainly consider the \textit{circular graph} $G\Box C_{2k}$, where $C_{2k}$ denotes the cycle on $2k$ vertices. Under the same condition as in Theorem \ref{Mohammadian} and for all fields $F$ of characteristic not 2, we obtain that $f(G\Box C_{2k})=n$ for all integers $k\geq 2$ (see Theorem \ref{circular} for details). Our main approach is to construct some $2k\times 2k$ partitioned matrices $R$ of order $nk$ and with corank $n$ by using blocks $B$ and $B^{-1}$ of order $n/2$,  which serve as weighted bi-adjacency matrices of $G \square C_{2k}$   (see Section 2). In Section 3 we  show that this algebraic method can be applied to non-balanced bipartite graphs $G$, provided their weighted bi-adjacency matrices are row-orthogonal or unitary, such as the complete bipartite graph $K_{m,n}$. Finally we give some further examples of applications of  Theorem \ref{circular} in Section 4.


\section{Minimum forcing number of circular graphs}

We now describe and prove our main result for the circular graph $G\Box C_{2k}$ as follows.


%

\begin{thm}\label{circular}
Let $G$ be a bipartite graph on $n$ vertices and $F$ a field of characteristic not 2. If $\mathcal{I}_F (G)\ne \emptyset$ and $k\ge2$, then
\begin{equation}\label{main}f(G\Box C_{2k})=n.\end{equation}
\end{thm}

\begin{proof} Let $(X, Y)$ be a bipartition of the bipartite graph $G$,  and let $(X_i, Y_i), i=1,2,\dots ,2k,$ be the  bipartition of the $i$-th  copy $G_i$ of $G$ in $G\Box C_{2k}$ corresponding to $(X, Y)$ such that  any pair of  corresponding vertices in $G_i$ and $G_{i+1}$ are adjacent, where the indices are taken modulo $2k$. Since $\mathcal{I}_F (G)\ne \emptyset$, $|X|=|Y|=n/2$.

First, we show that
\begin{equation}\label{ineq}f(G\Box C_{2k})\le n. \end{equation}
Take a matching $M$ in $G\Box C_{2k}$ that consists of the edges between $X_1$ and  $X_2$ and the edges between $Y_1$  and $Y_{2k}$. Let $G':=G\Box C_{2k}-(X_1\cup X_2\cup Y_1\cup Y_{2k})$. Then  each vertex in $Y_2$ has a unique neighbor in the  graph $G'$, so every perfect matching of $G'$ contains the matching between $Y_2$ and $Y_3$. Similarly, every perfect matching of $G'-(Y_2\cup Y_3)$ contains the matching between $X_3$ and $X_4$.  In this way  we can see that $M$ can be   uniquely extended to a perfect matching in $G\Box C_{2k}$. This implies that $f(G\Box C_{2k})\le |M|= n$.

Next, we show that $f(G\Box C_{2k})\ge n$. By  \cite{Hammack},   $G\Box C_{2k}$ is   bipartite. Since $\mathcal{I}_F (G)\ne \emptyset$,  take $A\in \mathcal{I}_F (G)$. Then $A$ has the following form:  $$A=\begin{pmatrix}
 O & B\\
 B^{-1} & O
\end{pmatrix},$$
which is  a matrix of even order $n$, where $B$ is a weighted bi-adjacency matrix of $G$ from $X$ to $Y$ and $B^{-1}$ is also a weighted bi-adjacency matrix of $G$ from $Y$ to $X$. We will construct a weighted bi-adjacency matrix $R$ of $G\Box C_{2k}$ that  is a $2k\times 2k$ partitioned matrix $(R_{ij})_{2k\times 2k}$, where each block $R_{ij}$ is a square matrix of order $n/2$. In particular, each block $I$ is an identity matrix of order $n/2$, and $2I$ is twice the identity matrix; the same notation applies to  $2B$ and $2B^{-1}$. For $i=1,2,\dots ,2k$, use $R^{(i)}$ to denote the $i$-th block row of $R$,  a $1\times 2k$ partitioned matrix.

{\bf Case 1:} $k =2 $.

Let $R$ be the following form of partitioned matrix:

\begin{equation}\label{R0}
   R:=\bordermatrix{
	& Y_1      &X_2     & Y_3      &X_4\cr
    X_1    & B      &I       &O      & 2I \cr
    Y_2    & I      & -B^{-1} &2I      & O\cr
    X_3    & O      & I       &-B      & I \cr
	Y_4    & I      &O        &I      & B^{-1}
    }.
\end{equation}
For the partitioned matrix $R$ in (\ref{R0}) we  can find the following two equalities of partitioned matrices:
$$R^{(1)}= B R^{(2)}+2I R^{(3)},\\
R^{(4)}= B^{-1} R^{(3)}+ R^{(2)},$$

\noindent which imply that the first $\frac{n}{2}$ row vectors and the last $\frac{n}{2}$ row vectors in matrix $R$ can be expressed as linear combinations of middle $n$ row vectors of $R$. Hence we have

$$ \operatorname{rank}(R)\le (4-2)\times \frac{n}{2}=n.$$

{\bf Case 2:} $k \equiv 0 \pmod{3} $.

Let $R$ be the following  partitioned matrix:
\begin{equation}\label{R1}
   R:=\bordermatrix{%
           & Y_1    & X_2	  & Y_3		 & X_4     &	\cdots	 & X_{2k-2}& Y_{2k-1}& X_{2k}\cr			
X_1   	   & B 		& I       & O   	 & O	   & \cdots      & O       & O       & I     \cr
Y_2   	   & I 		& B^{-1}  & I    	 & O	   & \cdots      & O       & O       & O     \cr
X_3   	   & O 		& I       & B   	 & I       & \cdots      & O       & O       & O     \cr
Y_4   	   & O 		& O       & I   	 & B^{-1}  & \cdots      & O       & O       & O     \cr
\vdots     &\vdots  & \vdots  &\vdots 	 & \vdots  & \ddots      &\vdots   &\vdots   & \vdots\cr
Y_{2k-1}   & O      & O       & O   	 & O       & \cdots      & B^{-1}  & I 		 & O    \cr
X_{2k-1}   & O      & O       & O   	 & O       & \cdots      & I       & B       & I     \cr
Y_{2k} 	   & I      & O       & O  	     & O       & \cdots      & O       & I       & B^{-1}
}.
\end{equation}

To describe $R$ in (\ref{R1}) more precisely, we define its $(i,j)$-blocks as follows:
\begin{itemize}
    \item The diagonal blocks satisfy $R_{ii} = B$ or $B^{-1}$ according to whether $i$ is odd or even.
    \item The off-diagonal blocks satisfy $R_{i-1,i} = R_{i+1,i} = I_{\frac{n}{2}}$ for $1 \le i \le 2k$, where indices are taken modulo $2k$.
    \item All other blocks are zero matrices of order $\frac{n}{2}$.
\end{itemize}

 It follows that each row (column) of $R$ has exactly three nonzero blocks.

We can obtain the following two equalities for partitioned matrices:
\begin{equation}\label{top1}
R^{(1)}= B R^{(2)} - B R^{(4)} + R^{(5)} + \sum_{i=1}^{\frac{k}{3}-1 } (-R^{(6i+1)} + B R^{(6i+2)} - B R^{(6i+4)} + R^{(6i+5)}),
\end{equation}
\begin{align}\label{bottom1}
R^{(2k)}= &B^{-1} R^{(2k-1)} - B^{-1} R^{(2k-3)} + R^{(2k-4)} + \sum_{i=1}^{\frac{k}{3}-1 } (-R^{(2k-6i)} +B^{-1} R^{(2k-6i-1)}\notag\\
&-B^{-1} R^{(2k-6i-3)}  + R^{(2k-6i-4)}).
\end{align}

Eqs. (\ref{top1}) and (\ref{bottom1}) imply that the first $\frac{n}{2}$ row vectors and the last $\frac{n}{2}$ row vectors in matrix $R$ can be expressed as linear combinations of middle $(kn-n)$ row vectors of $R$. Hence, we conclude
$$\operatorname{rank}(R)\le kn-n.$$

To confirm Eq. (\ref{top1}), let
$$Z=(Z_1,\dots ,Z_{2k})=\sum_{i=0}^{\frac{k}{3}-1 } (-R^{(6i+1)} + B R^{(6i+2)} - B R^{(6i+4)} + R^{(6i+5)}),$$
where $$Z_j=\sum_{i=0}^{\frac{k}{3}-1 } (-R_{6i+1,j} + B R_{6i+2,j} - B R_{6i+4,j} + R_{6i+5,j}), 1\le j\le 2k.$$

\noindent Next we  can show that each  block $Z_j$  is a zero matrix by some routine computations, which yields Eq. (\ref{top1}). By an analogous argument or the symmetry, we can confirm Eq. (\ref{bottom1}).

If $j=6t+1$, then

$$Z_j=\begin{cases}
OR_{2k,1}-R_{1,1} + BR_{2,1}= -B+BI=O,~~~~\mbox{for } t=0,\\
OR_{j-1,j}-R_{j,j} + B R_{j+1,j}= -B+BI=O,~~~~\mbox{for } 1\le t\le \frac{k}{3} -1.
\end{cases}$$

If $j=6t+2$ and $0\le t\le \frac{k}{3} -1$, then
$$Z_j= - R_{j-1,j}+ B R_{j,j} + O R_{j+1,j}= -I+BB^{-1}=O.$$

If $j=6t+3$ and $0\le t\le \frac{k}{3} -1$, then
$$Z_j= B R_{j-1,j}+ O R_{j,j} -B R_{j+1,j}= BI-BI=O.$$

If $j=6t+4$ and $0\le t\le \frac{k}{3} -1$, then
$$Z_j= O R_{j-1,j}- B R_{j,j} + R_{j+1,j}= -BB^{-1}+I=O.$$

If $j=6t+5$ and $0\le t\le \frac{k}{3} -1$, then
$$Z_j= -B R_{j-1,j}+ R_{j,j} +O R_{j+1,j}= -BI+ B=O.$$

If $j=6t+6$, then
$$Z_j=\begin{cases}
R_{j-1,j} + O R_{j,j} -R_{j+1,j}= I-I=O,~~~~\mbox{for }0\le t\le \frac{k}{3}-2,\\
R_{2k-1,2k}+ OR_{2k,2k}-R_{1,2k}= I-I=O,~~~~\mbox{for }t= \frac{k}{3} -1.
\end{cases}$$


{\bf Case 3:} $k \equiv 1 \pmod{3}$ and $k\geq 4 $.

Let $R$ be the following partitioned matrix:
\begin{equation}\label{R2}
R:=   \bordermatrix{%
           & Y_1    & X_2	  & Y_3	  & X_4		 & Y_5	        &\cdots     & X_{2k-4}  & Y_{2k-3}  & X_{2k-2}   & Y_{2k-1}   & X_{2k}\cr			
X_1   	   & B 		& 2I      & O     & O 		 & O              &\cdots             & O         & O         & O          & O          & -I    \cr
Y_2   	   & I 		& B^{-1}  & I     & O        & O               &\cdots            & O         & O         & O          & O          & O     \cr
X_3   	   & O 		& I       & B     & I        & O               &\cdots             & O         & O         & O          & O          & O     \cr
Y_4   	   & O 		& O       & 2I    & B^{-1}   & I               &\cdots             & O         & O         & O          & O          & O     \cr
X_5   	   & O 		& O       & O     & I        & B               &\cdots            & O         & O         & O          & O          & O     \cr
\vdots     &\vdots  & \vdots  &\vdots & \vdots   &\vdots      &\ddots         & \vdots    &\vdots     & \vdots     &\vdots      & \vdots\cr
Y_{2k-4}   & O 		& O       & O     & O        & O              &\cdots             & B^{-1}    & I         & O          & O          & O     \cr
X_{2k-3}   & O 		& O       & O     & O        & O              &\cdots             & I         & B         & 2I         & O          & O     \cr
Y_{2k-2}   & O 		& O       & O     & O        & O               &\cdots             & O         & I         & B^{-1}     & I          & O     \cr
X_{2k-1}   & O 		& O       & O     & O        & O               &\cdots             & O         & O         & I          & B          & I     \cr
Y_{2k} 	   & -I		& O       & O     & O        & O              &\cdots             & O         & O         & O          & 2I         & B^{-1}
},\end{equation}
where  $R_{ii}=B$ or $B^{-1}$ according to whether $i$ is odd or even, and $R_{i-1,i}=R_{i+1,i}=I_{\frac{n}{2}}$, where $1\le i \le 2k$ and the subscripts are taken modulo $2k$,  but except for $R_{1,2k}=R_{2k,1}=-I_{\frac{n}{2}}$ and $R_{1,2}=R_{4,3}=R_{2k,2k-1}=R_{2k-3,2k-2}=2I_{\frac{n}{2}}$, and the other blocks are zeroes. Similar to  (\ref{R1}) each row (column) of the partitioned matrix $R$ in (\ref{R2}) has exactly three nonzero blocks.

We can find the following equalities for partitioned matrices:
\begin{equation}\label{top2}
 R^{(1)}= B R^{(2)} + R^{(3)} - B R^{(4)} + \sum_{i=1}^{\frac{k-4}{3} } (B R^{(6i)}-R^{(6i+1)} + R^{(6i+3)} - B R^{(6i+4)}) + B R^{(2k-2)} - R^{(2k-1)},
\end{equation}
\begin{align}\label{bottom2}
 R^{(2k)}=&B^{-1} R^{(2k-1)} + R^{(2k-2)} - B^{-1} R^{(2k-3)} + \sum_{i=1}^{\frac{k-4}{3}} ( B^{-1} R^{(2k-6i+1)}-R^{(2k-6i)} + R^{(2k-6i-2)}\notag\\
 &-B^{-1} R^{(2k-6i-3)} ) + B^{-1} R^{(3)}- R^{(2)}.
\end{align}
Eqs. (\ref{top2}) and (\ref{bottom2}) imply that the first $\frac{n}{2}$ row vectors and the last $\frac{n}{2}$ row vectors in matrix $R$ can be expressed as linear combinations of middle $kn-n$ row vectors of $R$. Hence, we have
$$\operatorname{rank}(R)\le kn-n.$$

To confirm Eq. (\ref{top2}), let
\begin{align}Z=&(Z_1,\dots ,Z_{2k})= -R^{(1)} + B R^{(2)} + R^{(3)} - B R^{(4)}\notag\\& + \sum_{i=1}^{\frac{k-4}{3} } (B R^{(6i)} - R^{(6i+1)} + R^{(6i+3)} -B R^{(6i+4)})+ B  R^{(2k-2)} - R^{(2k-1)}.
\end{align}
For each $1\le j\le 2k$, we compute the $j$-th block $Z_j$:
$$Z_j=-R_{1,j} + B R_{2,j} + R_{3,j} - B R_{4,j} + \sum_{i=1}^{\frac{k-4}{3} } (B R_{6i,j} - R_{6i+1,j} + R_{6i+3,j} -B R_{6i+4,j})+ B  R_{2k-2,j} - R_{2k-1,j}.$$



If $j=6t+1$, then
$$Z_j=\begin{cases}
-R_{1,1} + B R_{2,1}= -B+BI=O, ~~\mbox{for }t=0, \\
B R_{j-1,j} - R_{j,j} + O R_{j+1,j}= BI - B=O, ~~\mbox{for } k> 4,1\le t\le \frac{k-4}{3},\\
B R_{2k-2,2k-1} - R_{2k-1,2k-1}=BI-B=O, ~~\mbox{for } j=2k-1=6\frac{k-1}{3}+1.
\end{cases}$$

If $j=6t+2$, then
$$Z_j=\begin{cases}
-R_{1,2} + B R_{2,2}+ R_{3,2}= -2I+BB^{-1}+I=O, ~~\mbox{for }t=0,\\
-R_{j-1,j} + O R_{j,j}+ R_{j+1,j}= -I+I=O,~~\mbox{for } k> 4,1\le t\le \frac{k-4}{3},\\
-R_{2k-1,2k}+ OR_{2k,2k}- R_{1,2k}=-I-(-I)=O, ~~\mbox{for } j=2k=6\frac{k-1}{3}+2.
\end{cases}$$

If $j=6t+3$, then
$$Z_j=\begin{cases}
B R_{2,3} + R_{3,3}-B R_{4,3}= BI+B -B(2I)=O, ~~\mbox{for }t=0,\\
O R_{j-1,j}+ R_{j,j} - B R_{j+1,j}=B-BI=O,~~\mbox{for } k> 4,1\le t\le \frac{k-4}{3}.
\end{cases}$$

If $j=6t+4$, then
$$Z_j= R_{j-1,j}- B R_{j,j} + O R_{j+1,j}=I -BB^{-1}=O, ~~\mbox{for }0\le t\le \frac{k-4}{3}.$$

If $j=6t+5$, then
$$Z_j= -B R_{j-1,j}+ O R_{j,j} +B R_{j+1,j}= -BI+ BI=O,~~\mbox{for } 0\le t\le \frac{k-4}{3}.$$

If $j=6t+6$, then
$$Z_j= O R_{j-1,j}+ B R_{j,j} - R_{j+1,j}=BB^{-1}-I=O,~~\mbox{for }0\le t\le \frac{k-4}{3}.$$

In short, $Z$ is a zero matrix, which yields Eq. (\ref{top2}). Similarly, we can confirm Eq. (\ref{bottom2}).

{\bf Case 4:} $k \equiv 2 \pmod{3}$ and $k\geq 5 $.

Let $R$ be the following  partitioned matrix:
\begin{equation}\label{R3}
R:=   \bordermatrix{%
           & Y_1    & X_2	  & Y_3	  & X_4		 & Y_5	   &\cdots	 & X_{2k-4} & Y_{2k-3} & X_{2k-2} & Y_{2k-1} & X_{2k} \cr			
X_1   	   & B 		& I       & O     & O 		 & O       &\cdots   & O        & O        & O        & O        & -I     \cr
Y_2   	   & I 		& B^{-1}  & I     & O        & O       &\cdots   & O        & O        & O        & O        & O      \cr
X_3   	   & O 		& I       & 2B    & I        & O       &\cdots   & O        & O        & O        & O        & O      \cr
Y_4   	   & O 		& O       & I     & B^{-1}   & I       &\cdots   & O        & O        & O        & O        & O      \cr
X_5   	   & O 		& O       & O     & I        & B       &\cdots   & O        & O        & O        & O        & O      \cr
\vdots     &\vdots  & \vdots  &\vdots & \vdots   &\vdots   &\ddots   &\vdots    &\vdots    & \vdots   &\vdots    & \vdots \cr
Y_{2k-4}   & O      & O       & O     & O        & O       &\cdots   & B^{-1}   & I        & O        & O        & O      \cr
X_{2k-3}   & O      & O       & O     & O        & O       &\cdots   & I        & B        & I        & O        & O      \cr
Y_{2k-2}   & O      & O       & O     & O        & O       &\cdots   & O        & I        & 2B^{-1}  & I        & O      \cr
X_{2k-1}   & O      & O       & O     & O        & O       &\cdots   & O        & O        & I        & B        & I      \cr
Y_{2k} 	   & -I     & O       & O     & O        & O       &\cdots   & O        & O        & O        & I        & B^{-1}
},\end{equation}
where  $R_{ii}=B$ or $B^{-1}$ according to whether $i$ is odd or even, but except for $R_{3,3}=2B$ and $R_{2k-2,2k-2}=2B^{-1}$, and $R_{i-1,i}=R_{i+1,i}=I_{\frac{n}{2}}$, where $1\le i \le 2k$ and the subscripts are taken modulo $2k$, but except for $R_{1,2k}=R_{2k,1}=-I_{\frac{n}{2}}$, and the other blocks are zeroes. Also each row (column) of the partitioned matrix $R$ in (\ref{R3}) has exactly three nonzero blocks.

We can obtain the following two equalities for partitioned matrices:
\begin{equation}\label{top3}
 R^{(1)}=\sum_{i=0}^{\frac{k-5}{3} } (B R^{(6i+2)} - B R^{(6i+4)}+R^{(6i+5)} - R^{(6i+7)}) + B  R^{(2k-2)} - R^{(2k-1)},
\end{equation}
\begin{equation}\label{bottom3}
R^{(2k)}=\sum_{i=0}^{\frac{k-5}{3}} ( B^{-1} R^{(2k-6i-1)}- B^{-1} R^{(2k-6i-3)}+ R^{(2k-6i-4)} - R^{(2k-6i-6)} ) + B^{-1} R^{(3)}- R^{(2)}.
\end{equation}

\noindent Eqs. (\ref{top3}) and (\ref{bottom3}) imply that the first $\frac{n}{2}$ row vectors and the last $\frac{n}{2}$ row vectors in matrix $R$ in (\ref{R3}) can be expressed as linear combinations of middle $kn-n$ vectors of $R$. Hence, we have
$$\operatorname{rank}(R)\le kn-n.$$

To confirm Eq. (\ref{top3}), let
$$Z=(Z_1,\dots ,Z_{2k})=-R^{(1)}+\sum_{i=0}^{\frac{k-5}{3}} (B R^{(6i+2)} - B R^{(6i+4)} + R^{(6i+5)}-R^{(6i+7)})+B R^{(2k-2)} -R^{(2k-1)}.$$
For each $1\le j\le 2k$, we compute the $j$-th block $Z_j$  as follows:
$$Z_j=\sum_{i=0}^{\frac{k-5}{3}} (-R_{6i+1,j}+ B R_{6i+2,j} - B R_{6i+4,j} + R_{6i+5,j})-R_{2k-3,j}+ B R_{2k-2,j}-R_{2k-1,j}.$$



%

If $j=6t+1$, then
$$Z_j=\begin{cases}
O R_{2k,1} -R_{1,1} + B R_{2,1} = -B+BI=O, ~~~~\mbox{for }t=0,\\
O R_{j-1,j}-R_{j,j} + B R_{j+1,j}= -B+BI=O,~~~~\mbox{for }1\le t\le \frac{k-5}{3},\\
O R_{2k-4,2k-3}- R_{2k-3,2k-3} +B R_{2k-2,2k-3}= -B+BI=O,~~~~\mbox{for }j=2k-3=6\frac{k-2}{3}+1,
\end{cases}$$

If $j=6t+2$, then
$$Z_j=\begin{cases} - R_{j-1,j}+ B R_{j,j} + O R_{j+1,j}= -I+BB^{-1}=O, ~~\mbox{for } 0\le t\le \frac{k-5}{3},\\
-R_{2k-3,2k-2}+ B R_{2k-2,2k-2} -R_{2k-1,2k-2}= -I+ B (2B^{-1})-I=O, ~~\mbox{for } j=2k-2=6\frac{k-2}{3}+2.
\end{cases}$$

If $j=6t+3$,  then
$$Z_j=\begin{cases} B R_{j-1,j}+ O R_{j,j} -B R_{j+1,j}= BI-BI=O,~~\mbox{for } 0\le t\le \frac{k-5}{3},\\
B R_{2k-2,2k-1}- R_{2k-1,2k-1}+O R_{2k,2k-1} = BI-B=O, ~~\mbox{for }j=2k-1=6\frac{k-2}{3}+3.
\end{cases}
$$

If $j=6t+4$,  then
$$Z_j= \begin{cases}O R_{j-1,j}- B R_{j,j} + R_{j+1,j}= -BB^{-1}+I=O,~~\mbox{for } 0\le t\le \frac{k-5}{3},\\
O R_{2k,2k}-R_{2k-1,2k}- R_{1,2k} =-I -(-I)=O, ~~\mbox{for }j=2k=6\frac{k-2}{3}+4.
\end{cases}
$$

If $j=6t+5$, then
$$Z_j= -B R_{j-1,j}+ R_{j,j} +O R_{j+1,j}= -BI+ B=O, ~~\mbox{for }0\le t\le \frac{k-5}{3}.$$

If $j=6t+6$, then
$$Z_j= R_{j-1,j}+ O R_{j,j} - R_{j+1,j}= I- I=O,~~\mbox{for }0\le t\le \frac{k-5}{3}.$$

In short, $Z$ is a  zero matrix, which  implies Eq. (\ref{top3}). Similarly, we can confirm Eq. (\ref{bottom3}).

Summarizing all the above, we have for all $k\ge 2$,
\begin{equation}\label{ineq2}\operatorname{rank}(R)\le kn-n, \text{and } \operatorname{corank}(R)\geq n. \end{equation}
Combining Ineq. (\ref{ineq2})   with Lemma \ref{lem} we have
$$f(G\Box C_{2k})\ge n,$$
which together with Ineq. (\ref{ineq}) imply   Eq. (\ref{main}). Hence the theorem holds.  
\end {proof}

Corollary 2.7 in \cite{Mohammadian} together with its proof and Theorem \ref{circular} imply the following result.

\begin{cor} Let $F$ be a field with $|F|\geq 4$ and $\operatorname{char}(F)\not=2$.  If $G$ is a bipartite graph with $\mathcal I_F(G)\not=\emptyset$, then $\mathcal I_F(G\Box Q_d)\not=\emptyset$ for any positive integer $d$; further, $f(G\Box Q_d) = 2^{d-2}|V(G)|$ and $f(G\Box Q_d\Box C_{2k}) = 2^{d}|V(G)|$ for any integer $k\geq 2$.
\end{cor}

\begin{Remark}\label{orth} If a bipartite graph $G$ admits an orthogonal matrix $B$ on a field $F$ or a unitary matrix $C$ on  the complex field $\mathbb{C}$ as a weighted bi-adjacency matrix, then \\ $$\begin{pmatrix}
 O & B\\
 B^{\top} & O
\end{pmatrix}\in\mathcal I_{F}(G)\not=\emptyset,\mbox{ or } \begin{pmatrix}
 O & C\\
C^{*} & O
\end{pmatrix}\in\mathcal I_{\mathbb{C}}(G)\not=\emptyset,$$
where $B^{\top}$ denotes the transpose of $B$ and the asterisk is used to denote the conjugate transpose.
\end{Remark}

 Mohammadian \cite{Mohammadian} gave some types of bipartite graphs which each admits an involutory weighted adjacency  matrix over real or complex fields: $d$-hypercube $Q_d$, folded $d$-hypercube $FQ_d$ (see also \cite{Alon}) for odd $d$, 2-blow-up of an even cycle $C_{2n}[2]$, complete bipartite graph $K_{n,n}$.  So by  Theorem \ref{circular} we obtain the minimum forcing number of their circular graphs as follows ($k\geq 2$).
\begin{itemize}
\item $f(Q_d\Box C_{2k})=2^d$.
  \item $f(FQ_d\Box C_{2k})=2^d$ for odd $d\geq 1$.
  \item $f(C_{2n}[2]\Box C_{2k})=4n$.
  \item $f(K_{n,n}\Box C_{2k})=2n$.
\end{itemize}

For the bipartite cocktail party graph $BCP(n)$ (obtained by removing  a perfect matching from $K_{n,n}$) for $n\not=1,3$, Mohammadian constructed an involutory weighted adjacency  matrix of $BCP(n)$ over a finite field $\mathbb F_p$, where $p$ is a prime number dividing $n-2$. Here $p$ may be equal to 2; in this case Theorem \ref{circular} does not apply. However, Bailey and Craigen \cite{Bailey} constructed a real orthogonal bi-adjacency  matrix of $BCP(n)$. In general, let $BCP(n,m)$ be a bipartite graph obtained from $K_{n,n}$ by removal of a  matching of size $m$, where $0\leq m\leq n$. Hence $BCP(n,n)=BCP(n)$ and $BCP(n,0)=K_{n,n}$.
\begin{thm}\cite[Theorem 5.8]{Bailey}\label{Bailey} Let $n\geq 1$ and $0 \le m \le n$. Then $BCP(n,m)$ has a real orthogonal bi-adjacency  matrix if and only if $(n,m) \notin \{(1,1), (2,1), (3,2), (3,3)\}.$
\end{thm}

By Theorems  \ref{Mohammadian}, \ref{circular} and  \ref{Bailey} and Remark \ref{orth}, we have the following result.
\begin{cor}Let $n\geq 1$ and $0 \le m \le n$. If $(n,m) \notin \{(1,1), (2,1), (3,2), (3,3)\} $, then $f(BCP(n,m)\Box K_2)=n$  and $f(BCP(n,m)\Box C_{2k})=2n$  for $k\geq 2$.
\end{cor}


\section{Minimum forcing number of prisms of graphs}
Theorem 2.10 in \cite{Mohammadian} states that for all positive integers $m$ and $n$, $f(K_{m,n}\Box K_2) = \min\{m,n\}$, which was proved by a  combinatorial method (some analysis  of  perfect matchings). However, for  $m\not=n$, Theorem \ref{Mohammadian} cannot be applied to directly obtain the result. In this section, we can generalize Theorem \ref{Mohammadian} as follows to handle prisms of non-balanced bipartite graphs.


Let $G=(X,Y)$ be a bipartite graph with $|X|=m, |Y|=n$~($m\le n$) and let $F$ be a field. We define
$\mathcal{R}_F (G)$ as the collection of pairs $(B,C)$, where  $B$, $C\in F^{m\times n}$  are weighted bi-adjacency matrices of $G$ satisfying $BC^{\top}=I_m.$

\begin{thm}\label{non-}
Let $G=(X,Y)$ be a bipartite graph with $|X|=m, |Y|=n$ and $m\le n$. If $\mathcal{R}_F (G)\ne \emptyset$, then
$$f(G\Box K_2)=m.$$
\end{thm}

\begin{proof}
Let  $(X_i, Y_i), i=1,2$, be the  bipartitions of the $i$-th  copy $G_i$ of $G$ in $G\Box K_2$ corresponding to $(X, Y)$.

First, we show that $f(G\Box K_2)\le m$. Take the matching $M$ in $G\Box K_2$ between $X_1$ and $X_2$. Then $|M|=m$, and $ G\Box K_2-(X_1\cup X_2)$ is formed by the matching between $Y_1$ and $Y_2$, which implies that $M$ can be uniquely extended to a perfect matching in $G\Box K_2$. This implies that $f(G\Box K_2)\le m$.

Next, it remains to show  $f(G\Box K_2)\ge m$. The graph $G\Box K_2$ is bipartite. Since  $ \mathcal{R}_F (G)\not=\emptyset$, there is a pair of weighted bi-adjacency matrices  $B$ and $C$  in $\mathcal{R}_F (G)$ such that  $BC^{\top}= I_m$. We construct a weighted bi-adjacency matrix $R$ of $G\Box K_2$ as follows.
$$
  R:= \bordermatrix{%
	& Y_2       &X_1\cr
	Y_1    & I_n      & C^{\top } \cr
	X_2    & B        & I_m
    }
.$$

\noindent Since
$$
   R=\begin{pmatrix}
 I_n & C^\top \\
 B & I_m
\end{pmatrix}=\begin{pmatrix}
 I_n\\B
\end{pmatrix}\begin{pmatrix}
 I_n & C^\top
\end{pmatrix},$$
we have $$
\operatorname{rank}(R)\le \operatorname{rank}\begin{pmatrix}
 I_n & C^\top
\end{pmatrix}=n.
$$

\noindent Thus, by Lemma \ref{lem} we have
$$f(G\Box K_2)\ge \frac{|G\Box K_2|}{2} -\operatorname{rank}(R)\ge m+n-n=m .$$
\end{proof}

\begin{ex} Consider star $K_{1,n}$.
 Then  $B=\frac{1}{\sqrt{n}}(1,\dots ,1)$ is a   weighted bi-adjacency matrix of $K_{1,n}$, and $(B,B)\in \mathcal{R}_{\mathbb{R} } (K_{1,n})$. By Theorem \ref{non-} we have $f(K_{1,n}\Box K_2)=1$.
\end{ex}

\begin{cor}\label{non-c}
Let $G=(X,Y)$ be a bipartite graph with $|X|=m, |Y|=n$ and $m\le n$. If $\mathcal{R}_F (G)\ne \emptyset$, and $H$ is a subgraph obtained from $G$ by deleting $k$ vertices in $X $, then
$$f(H\Box K_2)=m-k.$$
\end{cor}

\begin{proof}
 Since $\mathcal{R}_F (G)\ne \emptyset$, there exists a pair of  weighted bi-adjacency matrices $B_{m\times n}$ and $C_{m\times n}$ in $\mathcal{R}_F (G)$  such that  $BC^{\top}= I_m$.
Let $B_{(m-k)\times n}$ and $C_{(m-k)\times n}$ be the submatrices obtained from $B_{m\times n}$ and $C_{m\times n}$ by deleting the  rows indexed by these deleted $k$ vertices in $X$, respectively. Then  $B_{(m-k)\times n}C^{\top}_{(m-k)\times n}= I_{m-k}$, which implies $(B_{(m-k)\times n},C_{(m-k)\times n})\in \mathcal{R}_F (H)$. By Theorem 3.1, we have $f(H\Box K_2)=m-k.$
\end{proof}

We now  use Theorem \ref{non-} or  Corollary \ref{non-c} to deduce the above-mentioned result as follows.
\begin{pro}[Mohammadian \cite{Mohammadian}]
For any integers $1\le m\le n, f(K_{m,n}\Box K_2)=m.$
\end{pro}
\begin{proof}Consider the Fourier matrix $ F_n =(\xi^{ij})_{0\leq i,j\leq n-1}$,
where $\xi =e^{\frac{2\pi \text{i}}{n} } $ is a primitive $n$th root of unity. Let $B=\frac{1}{\sqrt{n}}F_n$. Since   $B$ is a unitary matrix, i.e.  $BB^*=I_n$, $\mathcal{R}_{\mathbb C} (K_{n,n})\ne \emptyset$.   By Corollary \ref{non-c}, we have $f(K_{m,n}\Box K_2)=n-(n-m)= m.$
\end{proof}

\begin{cor}
Let   $\{G_i = (X_i, Y_i)\}_{i=1}^t$ be a family of bipartite graphs with $|X_i| \le |Y_i|$,  $Y_i \cap Y_j = \emptyset$ for all $i \ne j$, and $X_i\cap (\cup_{j=1}^t Y_j) =\emptyset$ for all $i$. Let $G = \cup_{i=1}^t G_i$. For $F\in \{\mathbb R,\mathbb C\}$, if  $\mathcal{R}_F (G_i)\ne \emptyset$ for each $i$, then
\begin{equation}\label{coro2}f(G\Box K_2)=|X_1\cup \cdots \cup X_t|.\end{equation}
\end{cor}
\begin{proof} It is easy to see that  $G=\cup_{i=1}^t G_i$ is a bipartite graph with a bipartition $(\cup_{i=1}^t X_i, \cup_{i=1}^t Y_i)$. By Theorem \ref{non-}, it suffices to show that $\mathcal{R}_F (G)\ne \emptyset$. We proceed by induction on $t$, the number of graphs in the union. For $t=1$, the result is trivial by Theorem \ref{non-}. Now suppose  $t\geq 2$. Let $G_2'=\cup G_{i=2}^t$ with a bipartition $(X_2',Y_2')$, where $X_2'=\cup_{i=2}^t X_i$ and $Y_2'=\cup_{i=2}^t Y_i$. By the induction hypothesis we have $\mathcal{R}_F (G_2')\ne \emptyset$. Take $(B_1,C_1)\in \mathcal{R}_F  (G_1)$, and $(B_2,C_2)\in \mathcal{R}_F  (G_2')$. Then  $B_1C_1^{\top}=I_{|X_1|}$ and $B_2C_2^{\top}=I_{|X_2'|}$.

If $X_1\cap X_2'= \emptyset$, construct two weighted bi-adjacency matrices of $G$:
$$
  B:= \bordermatrix{%
	& Y_1       &Y_2'\cr
	X_1    & B_1  & O\cr
	X_2'    & O        &  B_2
    },
  C:=  \bordermatrix{%
	& Y_1       &Y_2'\cr
	X_1    & C_1  & O\cr
	X_2'    & O       &  C_2
    }.
$$
Then  $$BC^\top=\begin{pmatrix}
	 B_1  & O\\
	 O    &  B_2
    \end{pmatrix}\begin{pmatrix}
	 C_1^{\top}  & O\\
	 O    &  C_2^{\top}
    \end{pmatrix}=\begin{pmatrix}
	 B_1C_1^{\top}  & O\\
	 O    &  B_2C_2^{\top}
    \end{pmatrix}=I_{|X_1|+|X_2'|}.$$
Hence $(B,C)\in \mathcal{R}_F(G)$, so $\mathcal{R}_F(G)\not=\emptyset$.

Now suppose $X_0:= X_1\cap X_2'\not=\emptyset$. For $i=1,2$, let $B_i(X)$ denote the submatrix of $B_i$ formed by the rows corresponding to vertices in $X$. Define a weighted bi-adjacency matrix $B$ of $G$ by
$$
   B:=\bordermatrix{%
	& Y_1        &Y_2'\cr
  X_1\setminus X_0 & B_1(X_1\setminus X_0)  & O\cr
  X_0 & \frac{1}{\sqrt{2}}B_1(X_0)      &\frac{1}{\sqrt{2}}B_2(X_0)  \cr
X_2'\setminus X_0& O          &  B_2(X_2'\setminus X_0)
    }.
$$
Similarly, define another weighted bi-adjacency matrix $C$ of $G$ by
$$
   C:=\bordermatrix{%
	& Y_1        &Y_2'\cr
  X_1\setminus X_0 & C_1(X_1\setminus X_0)  & O\cr
  X_0 & \frac{1}{\sqrt{2}}C_1(X_0)      &\frac{1}{\sqrt{2}}C_2(X_0)  \cr
X_2'\setminus X_0& O          &  C_2(X_2'\setminus X_0)
    }.
$$
A direct computation  using $B_1C_1^{\top}=I_{|X_1|}$ and $B_2C_2^{\top}=I_{|X_2'|}$ yields   $BC^{\top}=I_{|X_1\cup \cdots \cup X_t|}$. Thus $(B,C)\in \mathcal{R}_F (G)$, so $\mathcal{R}_F (G)\ne \emptyset$, which implies Eq. (\ref{coro2}) by Theorem \ref{non-}.
\end{proof}

\section{Further examples}

In this section, we present additional examples illustrating the application of  Theorems \ref{Mohammadian} and \ref{circular}.

\begin{ex}

The signed graph $S_{14}$  \cite{Mckee} is  a 4-regular bipartite graph with 14 vertices $x_0, x_1,\dots, x_6$, $ y_0, y_1, \dots, y_6$, where each $x_i$ is connected to  $y_i, y_{i+1}$ and $y_{i+3}$ by positive edges, and to $y_{i-1}$ by a negative edge (the subscripts modulo 7).  Assigning weight  $\frac{1}{2}$ to each positive edge and $-\frac{1}{2}$ to each negative edge, we obtain  a weighted bi-adjacency matrix $B$ of $S_{14}$  as shown in Figure 1, where the rows are indexed by $x_0, x_6, x_4, x_1,x_3,x_5$ and the columns by  $y_0, y_1, y_3, y_6,y_4,y_5,y_2$. The matrix $B$ is an orthogonal matrix of order 7. Hence,   $\mathcal{I}_{\mathbb R} (S_{14})\ne \emptyset$. By Theorems \ref{Mohammadian} and \ref{circular}, we have $f(S_{14}\Box K_2)=7$ and $f(S_{14}\Box C_{2k})=14$ for  $k\ge2.$

\begin{figure}[thbp!]
    \centering
    \begin{minipage}[l]{0.49\linewidth}
        \includegraphics[width=0.8\linewidth]{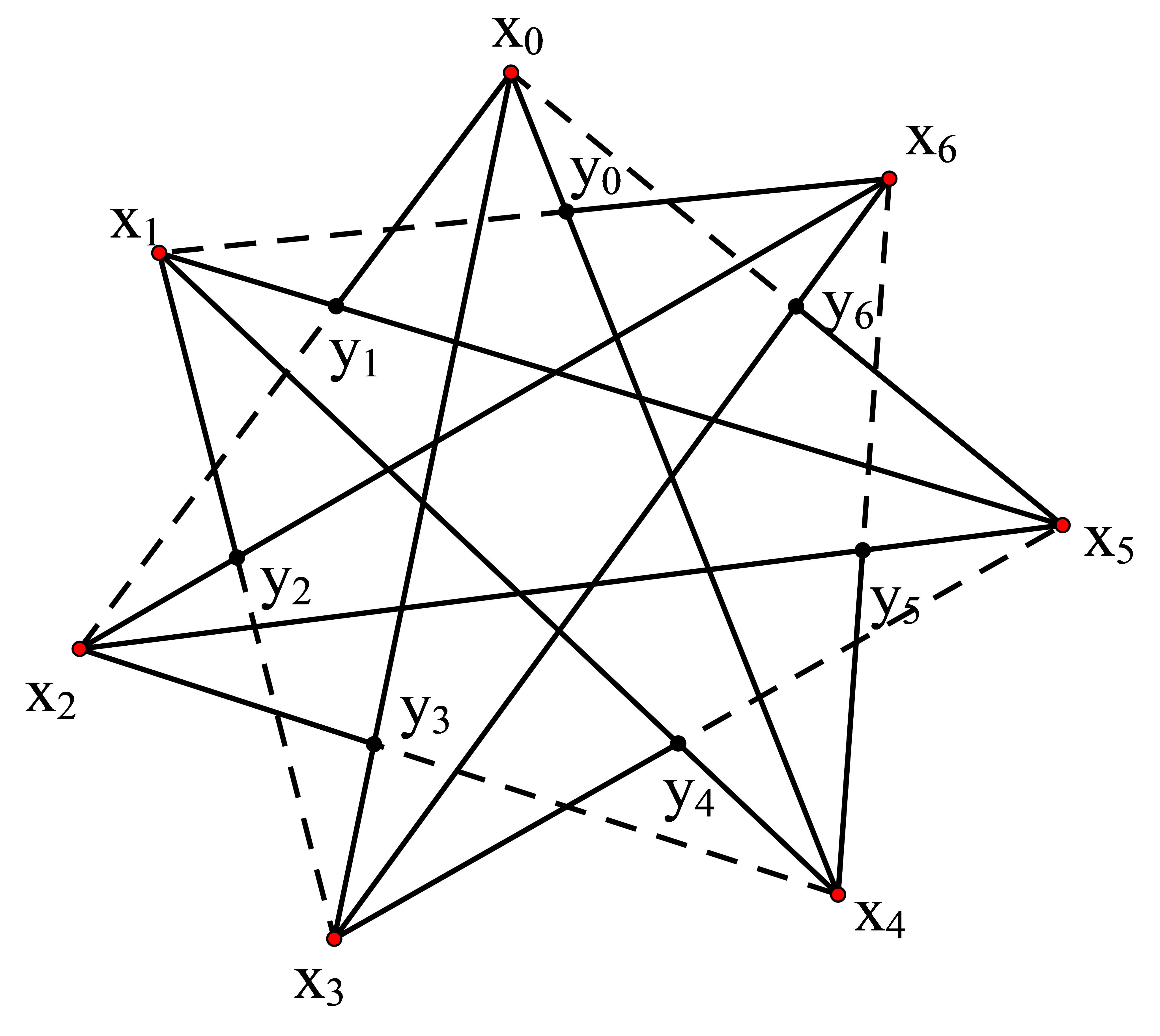}
        \end{minipage}
    \begin{minipage}[r]{0.49\linewidth}
        $\begin{pmatrix}
 \frac{1}{2} 	& \frac{1}{2} & \frac{1}{2} & -\frac{1}{2} & 0 & 0 & 0\\
 \frac{1}{2} 	& 0 & 0 & \frac{1}{2} & 0 & \frac{1}{2} & -\frac{1}{2}\\
 \frac{1}{2} 	& 0 & -\frac{1}{2} & 0 & \frac{1}{2} & 0 & \frac{1}{2}\\
 -\frac{1}{2} & \frac{1}{2} & 0 & 0 & \frac{1}{2} & \frac{1}{2} & 0\\
 0 & 0		   & \frac{1}{2} & \frac{1}{2} & \frac{1}{2} & -\frac{1}{2} & 0\\
 0 & \frac{1}{2} & 0 & \frac{1}{2} & -\frac{1}{2} & 0 & \frac{1}{2}\\
 0 & -\frac{1}{2} & \frac{1}{2} & 0 & 0 & \frac{1}{2} & \frac{1}{2}
       \end{pmatrix}$
   \end{minipage}
\caption{\label{essential} Signed graph  $S_{14}$ with a weighted bi-adjacency matrix (see also \cite{Mckee}).}
 \end{figure}
\end{ex}

The direct product of two graphs $G$ and $H$, denoted by $G\times H$, is the graph with vertex set $V(G)\times V(H)$, and two vertices $(g_1,h_1)$ and $(g_2,h_2)$ are adjacent if and only if both $g_1g_2 \in E(G)$ and $h_1h_2 \in E(H)$. In particular, $G\times K_2$ is called the bipartite double of $G$, denoted by $bd(G)$.

\begin{ex}

Consider  a family of  5-regular graphs illustrated in Figure \ref{essential2}. The left graph $G_1$ has $4t+2$ vertices and the right graph $G_2$ has $4t+4$ vertices, $t\geq 2$.  
Stani\'c \cite{Stani} pointed out that the  bipartite doubles $bd(G_i)$ of $G_i$ are  bipartite connected graphs and showed that $bd(G_1)$ and $bd(G_2)$ admit  weighted bi-adjacency matrices  $B_1:=E(4t+2,5)$ and  $B_2:=F(4t+4,5)$ as shown in Figure \ref{essential3}, where  ``+" and ``$-$"  denote ``1" and ``$-1$", respectively, and the empty positions indicate zero entries (see Lemma 3.10 in \cite{Stani}). Since these matrices are row orthogonal,
$$\frac{ 1}{\sqrt{5}}\begin{pmatrix}
 O & B_1\\
 B_1^\top & O
\end{pmatrix}\in \mathcal{I}_{\mathbb R} (bd(G_1)), \quad
\frac{ 1}{\sqrt{5}}\begin{pmatrix}
 O & B_2\\
 B_2^\top & O
\end{pmatrix}\in \mathcal{I}_{\mathbb R} (bd(G_2)).$$
Hence, we have $f(bd(G_1)\Box K_2)=4t+2$ and $f(bd(G_2)\Box K_2)=4t+4$ by Theorem \ref{Mohammadian}, and $f(bd(G_1)\Box C_{2k})=8t+4$ and $f(bd(G_2)\Box C_{2k})=8t+8$ for $k\geq 2$ by Theorem \ref{circular}.

\begin{figure}[!htbp]
\begin{center}
\includegraphics[totalheight=4cm]{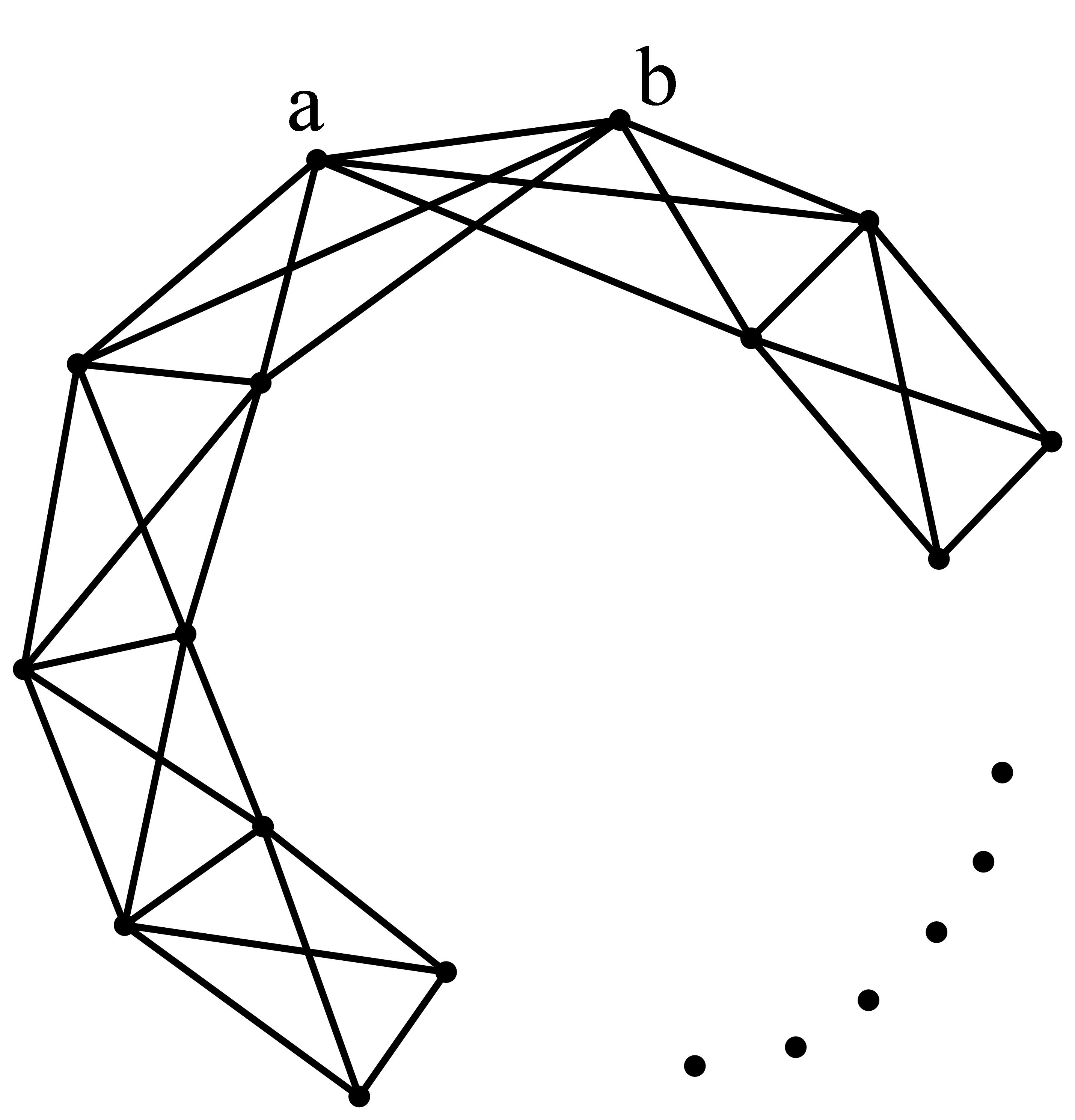}~~~~~~~
\includegraphics[totalheight=4cm]{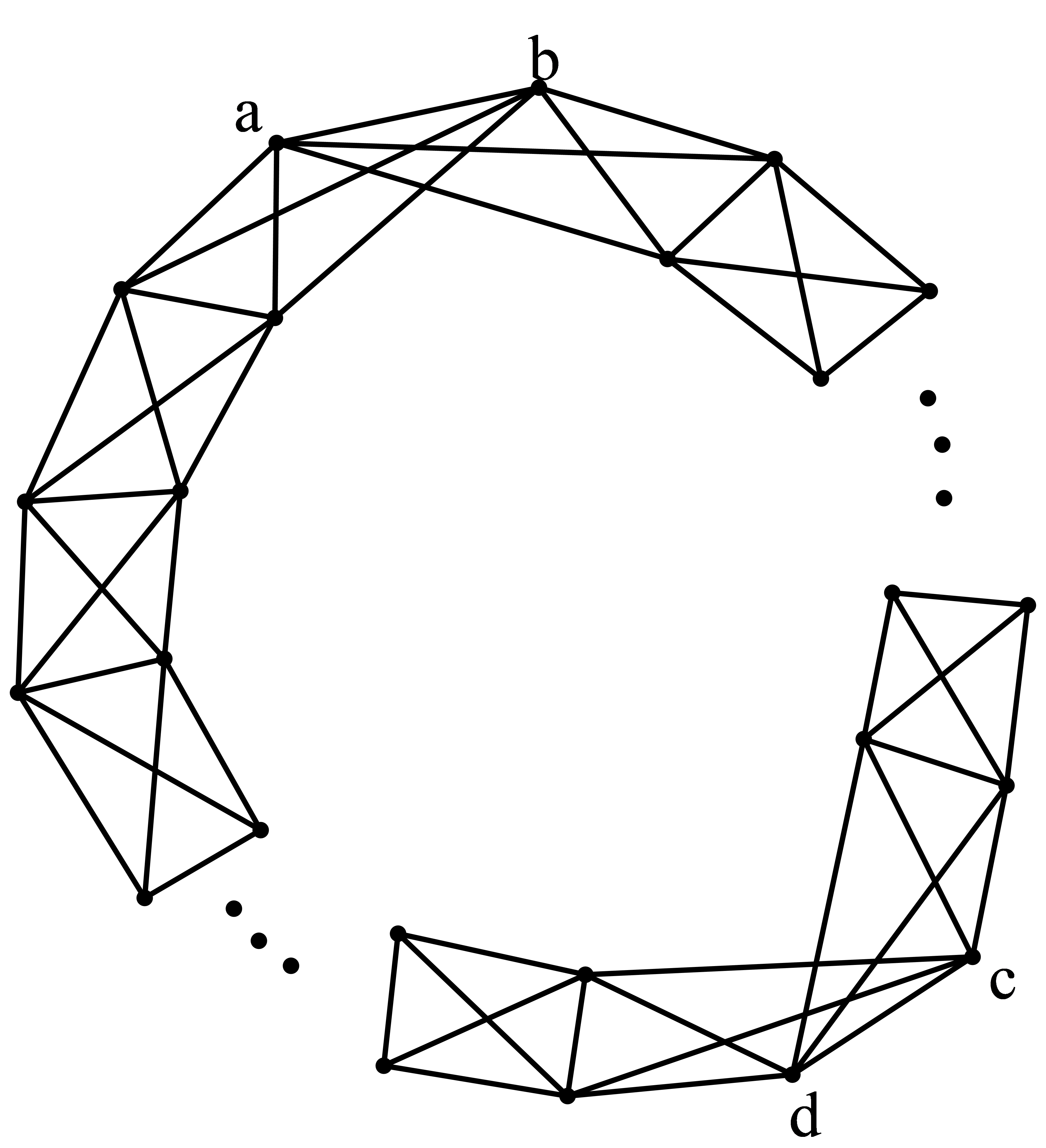}
\caption{\label{essential2} Illustration for graphs $G_1$ (left) and $G_2$ (right).}
\end{center}
\end{figure}

\begin{figure}[!htbp]
\begin{center}
\tiny{
\addtocounter{MaxMatrixCols}{10}
\setlength{\arraycolsep}{1.7pt}
\renewcommand\arraystretch{0.9}
$\begin{pmatrix}
 + & + & + & + & + &  &  &  &  &  &  &  &  &  &  &  &  &  & \\
 + & + & - & - &  & + &  &  &  &  &  &  &  &  &  &  &  &  & \\
 + & - &  &  &  &  & + & + & + &  &  &  &  &  &  &  &  &  & \\
 + & - &  &  &  &  & - & - &  & + &  &  &  &  &  &  &  &  & \\
 + &  &  &  & - & - &  &  & + & + &  &  &  &  &  &  &  &  & \\
   & + &  &  & - & - &  &  & + & + &  &  &  &  &  &  &  &  & \\
   &  & + & - &  &  &  &  &  &  &  &  &  &  &  &  &  &  & \\
   &  & + & - &  &  &  &  &  &  &  &  &  &  &  &  &  &  & \\
   &  & + &  & - & + &  &  &  &  &  &  &  &  &  &  &  &  & \\
   &  &  & + & - & + &  &  &  &  &  &  &  &  &  &  &  &  & \\
   &  &  &  &  &  &  &  &  &  & \ddots  &  &  &  &  &  &  &  & \\
   &  &  &  &  &  &  &  &  &  &  &  &  &  &  & + & + & + & \\
   &  &  &  &  &  &  &  &  &  &  &  &  &  &  & - & - &  & +\\
   &  &  &  &  &  &  &  &  &  &  &  &  &  &  &  &  & - & -\\
   &  &  &  &  &  &  &  &  &  &  &  &  &  &  &  &  & + & +\\
   &  &  &  &  &  &  &  &  &  &  & + & - &  &  & + &  & - & +\\
   &  &  &  &  &  &  &  &  &  &  & + & - &  &  &  & - & + & -\\
   &  &  &  &  &  &  &  &  &  &  & + &  & - & + & - & + &  & \\
   &  &  &  &  &  &  &  &  &  &  &  & + & - & + & + & - &  &
\end{pmatrix}$
\addtocounter{MaxMatrixCols}{10}
\setlength{\arraycolsep}{1.7pt}
\renewcommand\arraystretch{0.9}
$\begin{pmatrix}
 + & + & + & + & + &  &  &  &  &  &  &  &  &  &  &  &  &  &  &  & \\
 + & + & - & - &  & + &  &  &  &  &  &  &  &  &  &  &  &  &  &  & \\
 + & - &  &  &  &  & + & + & + &  &  &  &  &  &  &  &  &  &  &  & \\
 + & - &  &  &  &  & - & - &  & + &  &  &  &  &  &  &  &  &  &  & \\
 + &  &  &  & - & - &  &  & + & + &  &  &  &  &  &  &  &  &  &  & \\
  & + &  &  & - & - &  &  & + & + &  &  &  &  &  &  &  &  &  &  & \\
  &  & + & - &  &  &  &  &  &  &  &  &  &  &  &  &  &  &  &  & \\
  &  & + & - &  &  &  &  &  &  &  &  &  &  &  &  &  &  &  &  & \\
  &  & + &  & - & + &  &  &  &  &  &  &  &  &  &  &  &  &  &  & \\
  &  &  & + & - & + &  &  &  &  &  &  &  &  &  &  &  &  &  &  & \\
  &  &  &  &  &  &  &  &  &  & \ddots  &  &  &  &  &  &  &  &  &  & \\
  &  &  &  &  &  &  &  &  &  &  &  &  &  &  & + & + & + &  &  & \\
  &  &  &  &  &  &  &  &  &  &  &  &  &  &  & - & - &  & + &  & \\
  &  &  &  &  &  &  &  &  &  &  &  &  &  &  &  &  & - & - &  & \\
  &  &  &  &  &  &  &  &  &  &  &  &  &  &  &  &  & + & + &  & \\
  &  &  &  &  &  &  &  &  &  &  & + & - &  &  & + & - &  &  & + & \\
  &  &  &  &  &  &  &  &  &  &  & + & - &  &  & - & + &  &  &  & +\\
  &  &  &  &  &  &  &  &  &  &  & + &  & - & + &  &  &  &  & - & -\\
  &  &  &  &  &  &  &  &  &  &  &  & + & - & + &  &  &  &  & + & +\\
  &  &  &  &  &  &  &  &  &  &  &  &  &  &  & + &  & - & + & - & +\\
  &  &  &  &  &  &  &  &  &  &  &  &  &  &  &  & + & - & + & + & -
\end{pmatrix}$
}

\caption{\label{essential3}$(0,1,-1)$-matrices $E(4t+2,5)$ (left) and $F(4t+4,5)$  (right), $t\geq 2$ (adapted from \cite{Stani}).}
\end{center}
\end{figure}
\end{ex}

\begin{ex}
All previous examples mainly concern regular graphs.   However, for a bipartite graph $G$, $\mathcal{I}_F (G)\ne \emptyset$ does not necessarily guarantee that $G$ is  regular, as Theorem 2.4 gives an example  in $BCP(n,m)$ with $1\leq m<n$. The graph $G'$ in  Figure \ref{essential4},  obtained by removing 8 edges from $K_{7,7}$, is an additional example, which admits a weighted bi-adjacency matrix of order 7, as shown in  Figure \ref{essential4} (right) (see Example 3.19 in \cite{Brennan}).  This matrix is square and row orthogonal, and its standardization yields an orthogonal matrix, which implies $\mathcal{I}_{\mathbb R} (G')\not=\emptyset$. Hence, by Theorems \ref{Mohammadian} and \ref{circular}, we have $f(G'\Box K_2)=7$ and   $f(G'\Box C_{2k})=14$ for $k\ge2.$

\begin{figure}[!htbp]
\begin{center}
    \centering
    \begin{minipage}[l]{0.49\linewidth}
        \includegraphics[width=0.8\linewidth]{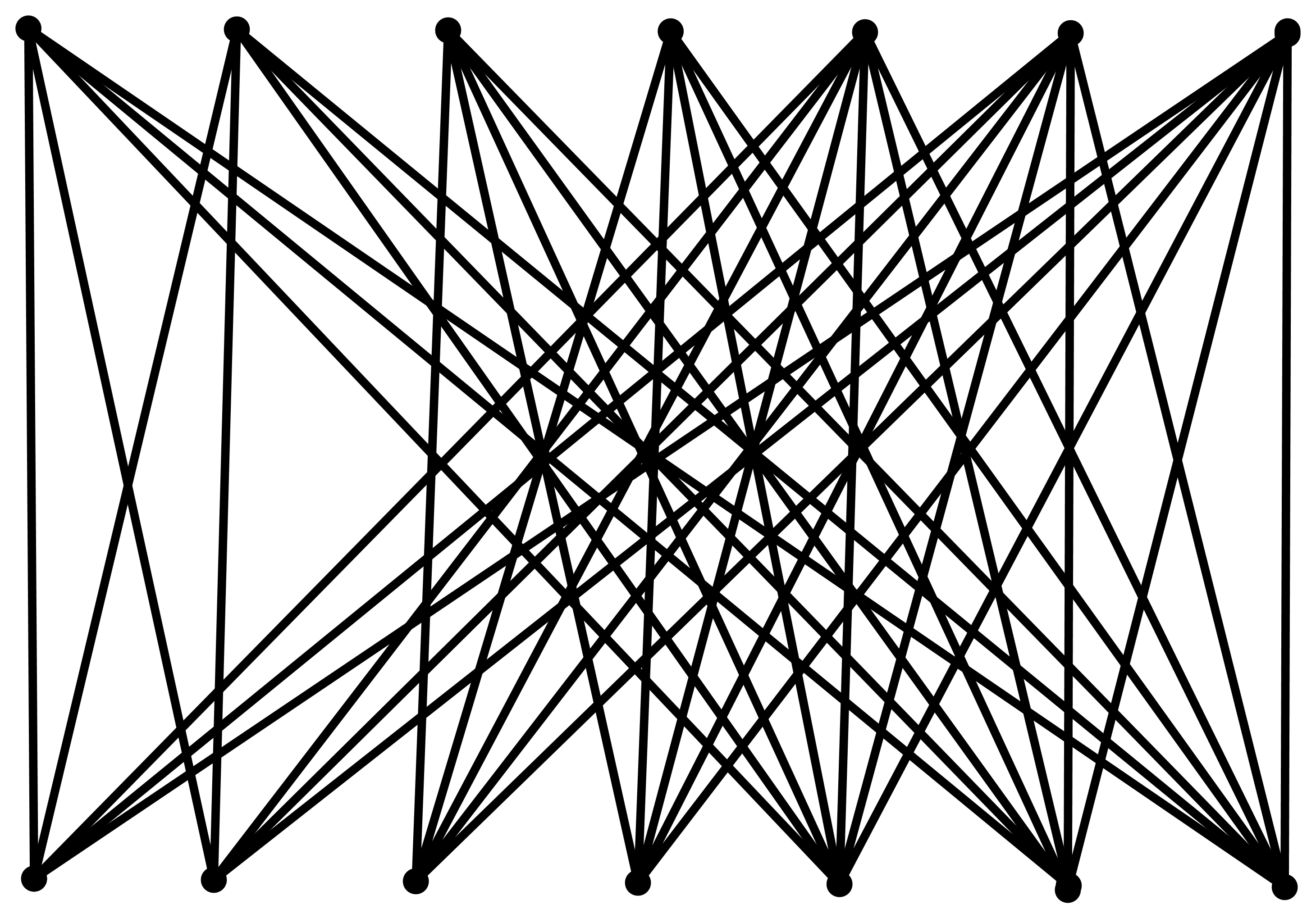}
        \end{minipage}
    \begin{minipage}[r]{0.49\linewidth}
        $\begin{pmatrix}
 -9 & 9 & 0 & 0 & 3\sqrt{2}  & -6\sqrt{2}  & 6\sqrt{2} \\
 9 & -9 & 0 & 0 & 3\sqrt{2}  & -6\sqrt{2}  & 6\sqrt{2} \\
 0 & 0 & -9 & 9 & -6\sqrt{2}  & 3\sqrt{2}  & 6\sqrt{2} \\
 0 & 0 & 9 & -9 & -6\sqrt{2}  & 3\sqrt{2}  & 6\sqrt{2} \\
 3\sqrt{2} & 3\sqrt{2} & -6\sqrt{2} & -6\sqrt{2} & 8 & 8 & 4 \\
 -6\sqrt{2} & -6\sqrt{2} & 3\sqrt{2} & 3\sqrt{2} & 8 & 8 & 4 \\
 6\sqrt{2} & 6\sqrt{2} & 6\sqrt{2} & 6\sqrt{2} & 4 & 4 & 2 \\
\end{pmatrix}$
   \end{minipage}
\caption{\label{essential4} A bipartite graph $G'$ with  a weighted bi-adjacency matrix (see \cite{Brennan}).}
\end{center}
\end{figure}

\end{ex}


\section*{Acknowledgments}

The authors are grateful to two anonymous referees for their  valuable comments and suggestions to improve this manuscript.


\end{document}